\documentclass[11pt,reqno,a4paper]{amsart}

\usepackage{amsmath,amssymb,amsthm,mathtools}
\usepackage{enumitem}
\usepackage{ragged2e}
\usepackage{microtype}
\usepackage{mathrsfs}
\usepackage{aliascnt}
\usepackage{hyperref}
\usepackage{fullpage}
\usepackage[nameinlink,capitalise,noabbrev]{cleveref}
\usepackage{lineno}
\usepackage{setspace}
\allowdisplaybreaks
\setlist{itemsep=2pt,topsep=4pt}
\numberwithin{equation}{section}

\hypersetup{
  colorlinks=true,
  linkcolor=blue,
  citecolor=blue,
  urlcolor=blue,
  pdftitle={Fractional Stability and Exact Minimum d-Degree Thresholds for Hypergraph Perfect Matchings},
  pdfsubject={Perfect matchings in uniform hypergraphs},
  pdfkeywords={hypergraph matching, minimum d-degree, fractional matching, space barrier, divisibility barrier, stability}
}

\newtheorem{theorem}{Theorem}[section]
\newaliascnt{lemma}{theorem}
\newtheorem{lemma}[lemma]{Lemma}
\aliascntresetthe{lemma}
\newaliascnt{proposition}{theorem}

\aliascntresetthe{proposition}
\newaliascnt{corollary}{theorem}
\newtheorem{claim}{Claim}

\aliascntresetthe{corollary}
\theoremstyle{definition}
\newaliascnt{definition}{theorem}
\newtheorem{definition}[definition]{Definition}
\aliascntresetthe{definition}
\theoremstyle{remark}
\newaliascnt{remark}{theorem}

\aliascntresetthe{remark}

\newcommand{\eps}{\varepsilon}

\newcommand{\ddiv}{\delta_{\mathrm{div}}}

\title[Exact minimum $d$-degree thresholds]
{Exact Minimum $d$-Degree Thresholds for Hypergraph Perfect Matchings}

\author{Jie Han}
\address{JH and BW. School of Mathematics and Statistics, Beijing Institute of Technology, China\\
Email: \texttt{(JH) han.jie@bit.edu.cn, (BW) bin.wang@bit.edu.cn.}}
\author{Hongliang Lu}
\address{HL and FY. School of Mathematics and Statistics
Xi'an Jiaotong University, Xi'an, China\\
Email:\texttt{(HL) luhongliang215@sina.com, (FY) fhyuan1@gmail.com.}}
\author{Bin Wang}
\author{Feihong Yuan}
\date{}

\subjclass[2020]{05C65, 05C70, 05D05, 60E15}
\keywords{perfect matching, uniform hypergraph, minimum $d$-degree, fractional matching, space barrier, divisibility barrier, stability}

\begin{document}

\begin{abstract}
For fixed integers $k\ge3$ and $1\le d\le k-1$ and sufficiently large $n\in k\mathbb N$, we establish the sharp minimum $d$-degree thresholds that forces perfect matching in every $n$-vertex $k$-uniform hypergraphs.
This was conjectued by Treglown and Zhao, and the $d=1$ case was conjectued by K\"uhn, Osthus and Treglown.
\end{abstract}

\maketitle

\section{Introduction}\label{sec:introduction}

Matchings are fundamental objects in graphs and hypergraphs, with various applications and connections with many fields of mathematics, computer science, and other branches of natural and social science.
Graph matchings are well understood -- Tutte's condition \cite{Tutte47} completely characterize the graphs with perfect matchings, and Edmonds' algorithm \cite{Edmonds65} finds a maximum matching in any graph in polynomial time.
In contrast, hypergraph matchings remain a less understood subject, and indeed the hypergraph perfect matching problem is one of the 21 NP-complete problems of Karp \cite{Karp72} in 1972.
In this paper we are interested in the extremal aspect of perfect matching in hypergraphs, that is, the minimum degree conditions.
The celebrated Dirac's theorem says that every $n$-vertex graph with minimum degree $n/2$ contains a spanning cycle; if $n$ is even, this gives a perfect matching.
Extending this result to hypergraphs has been the main focus of the Dirac-type problems for the past two decades.

\subsection{Minimum degrees and the two barriers}

Let $H$ be a $k$-uniform hypergraph, or briefly a $k$-graph.  
A \emph{matching} in
$H$ is a family of pairwise disjoint edges, and it is \emph{perfect} if it
covers every vertex. 
Let $\nu(H)$ denote the maximum size of a matching in $H$.
For a set $D\in\binom{V(H)}d$, write
$\deg_H(D):=\bigl|\{e\in E(H):D\subseteq e\}\bigr|$, $\delta_d(H):=\min_{D\in\binom{V(H)}d}\deg_H(D)$.
When $n\in k\mathbb N$, let $m_d(k,n)$ denote the least integer $m$ such that every
$n$-vertex $k$-graph $H$ with $\delta_d(H)\ge m$ contains a perfect matching.
In this paper we concern the behavior of $m_d(k,n)$ when $n$ is sufficiently large with respect to $k$.
The threshold is governed by two different obstructions.

The first is the \emph{space barrier}.  Choose a set $A$ of size $n/k-1$ and
let $\mathcal S(A):=\{e\in\binom{V}{k}:e\cap A\ne\varnothing\}$.
Every matching in $\mathcal S(A)$ has at most $|A|$ edges.  A $d$-set
contained in $V\setminus A$ has the smallest degree, so
$s_d(k,n):=\delta_d(\mathcal S(A))
 =\binom{n-d}{k-d}-\binom{n-d-n/k+1}{k-d}$.
 Its asymptotic density is
$c_{k,d}^{\mathrm{sp}}
 :=1-\left(1-\frac1k\right)^{k-d}$.
The second obstruction is a parity, or \emph{divisibility barrier}.
Given a partition $V=A\dot\cup B$ and $q\in\{0,1\}$, define $\mathcal P(A,q):=
 \{e\in\binom Vk:|e\cap A|\equiv q\pmod2\}$.
If $|A|\not\equiv qn/k\pmod2$, then $\mathcal P(A,q)$ has no perfect
matching.  
Let
$\ddiv(n,k,d):=\max\bigl\{\delta_d(\mathcal P(A,q)):
 A\subseteq V,\ q\in\{0,1\},\ |A|\not\equiv q(n/k)\pmod2\bigr\}$.
Consequently,
\begin{equation}\label{eq:two-lower-bounds}
 m_d(k,n)\ge\max\{s_d(k,n),\ddiv(n,k,d)\}+1.
\end{equation}
For fixed $k$ and $d$,
\begin{equation}\label{eq:barrier-asymptotics}
 \frac{s_d(k,n)}{\binom{n-d}{k-d}}
 =c_{k,d}^{\mathrm{sp}}+o(1),
 \quad \text{ and } \quad
 \frac{\ddiv(n,k,d)}{\binom{n-d}{k-d}}
 =\frac12+o(1).
\end{equation}
The two limiting densities never coincide.  Indeed, with $r=k-d$, the
identity $c_{k,d}^{\mathrm{sp}}=1/2$ would imply
$2(k-1)^r=k^r$; since $\gcd(k,k-1)=1$, this would force $2\in k^r\mathbb N$, which is
impossible for $k\ge3$. 
Thus every pair $(k,d)$ lies in exactly one of the
space-dominant and divisibility-dominant regimes.

Our main theorem proves that the lower bound in
\eqref{eq:two-lower-bounds} is always exact, which resolves a conjecture of Treglown and Zhao \cite{TZ16} (the $d=1$ case was conjectured earlier by K\"uhn, Osthus and Treglown \cite{KOT13}).

\begin{theorem}[Exact minimum $d$-degree threshold]\label{thm:exact-threshold}
Fix $k\ge3$ and $d\in[k-1]$.  For every sufficiently large
$n\in k\mathbb N$,
$m_d(k,n)=\max\{s_d(k,n),\ddiv(n,k,d)\}+1$.
\end{theorem}

Indeed, the divisibility-dominant range is already handled by Theorem 2 of Treglown and Zhao~\cite{TZ16}, and our work focus on the space-dominant case, e.g., $c_{k,d}^{\mathrm{sp}} = 1-(1-1/k)^{k-d} > 1/2$.

\subsection{Fractional stability and the space-dominant theorem}
The following fractional relaxation has been proven to be closely related to the integer matching problem, and is also essential to us.
A \emph{fractional matching} in $H$ is a function
$\varphi:E(H)\to[0,1]$ such that $\sum_{e\ni v}\varphi(e)\le1$ for every vertex $v$.
Its \emph{weight} is $\sum_{e\in E(H)}\varphi(e)$, and it is \emph{perfect} if its weight is
$|V(H)|/k$.  
The \emph{fractional matching number} of $H$ is $\nu^*(H):=
\max\{
\sum_{e\in E(H)}\varphi(e):
\varphi \text{ is a fractional matching of }H\}$.
A \emph{fractional vertex cover} is a function
$w:V(H)\to[0,1]$ such that $\sum_{v\in e}w(v)\ge1$ for every edge $e$.
The \emph{fractional vertex-cover number} of $H$ is $\tau^*(H):=
\min\{
\sum_{v\in V(H)}w(v):
w \text{ is a fractional vertex cover of }H\}$.
Linear programming duality gives $\nu^*(H)=\tau^*(H)$.

Let $m_d^*(k,N)$ be the least minimum $d$-degree that forces a fractional
perfect matching in every $N$-vertex $k$-graph, and define
\begin{equation}\label{eq:fractional-constant}
 c^*_{k,d}:=
 \limsup_{N\to\infty}
 \frac{m_d^*(k,N)}{\binom{N-d}{k-d}}.
\end{equation}
The space construction gives $c^*_{k,d}\ge c_{k,d}^{\mathrm{sp}}$.
Alon, Frankl, Huang, R\"odl, Ruci\'nski and Sudakov~\cite{AFHRRS}
conjectured equality holds for every fixed pair $(k,d)$. 
This conjecture is now established via its beautiful connection to Feige's conjecture in probability theory.
We only state a specialized form sufficient for our purpose.
Let $X_1,\ldots,X_\ell$ be independent, non-negative random variables, each of which has mean $1$. Then the conjecture says that for all $d\ge 1$,
\[
\Theta^d(\ell)=\sup\Pr[X_1+\cdots+X_\ell\ge \ell+d]
=
1-\left(1-\frac{1}{\ell+d}\right)^\ell.
\]
This is verified very recently by Fu, Han, Wang, Yan, Zhang and Zhou~\cite{FHWYZZ26} and independently by Nie and Wei \cite{NW26}.
Then Theorem 1.14 of Ferber and Jain \cite{FJ19} says that $c_{k,d}^*=\Theta^d(k-d) = 1-\left(1-\frac{1}{k}\right)^{k-d}$.



\begin{theorem}[Exact fractional threshold, \cite{FHWYZZ26, NW26}+\cite{FJ19}]\label{cor:fractional-threshold}
Fix $k\ge3$ and $d\in[k-1]$, $c^*_{k,d}=c_{k,d}^{\mathrm{sp}}$.
\end{theorem}

In our proof we indeed need a stability version of this result, which, fortunately, is established by Cao, Liu and Zhang \cite{CLZ26}, again, very recently (we do use the statement of Theorem \ref{cor:fractional-threshold} for numerical use).
The second new component treats the range in which the space barrier has
limiting density greater than $1/2$.

\begin{theorem}[Space-dominant exact threshold]\label{thm:space-dominant}
Fix $k\ge3$ and $d\in[k-1]$, and suppose that $1-(1-1/k)^{k-d}>\frac12$.
Then, for every sufficiently large $n\in k\mathbb N$, $m_d(k,n)=s_d(k,n)+1
 =\binom{n-d}{k-d}-
  \binom{n-d-n/k+1}{k-d}+1$.
\end{theorem}

For comparison, Theorem~2 of Treglown and Zhao~\cite{TZ16} gives
\[
 m_d(k,n)=
 \max\left\{
 \ddiv(n,k,d)+1,
 \bigl(c^*_{k,d}+o(1)\bigr)\binom{n-d}{k-d}
 \right\},
\]
which solves the parity-dominant branch by Theorem \ref{cor:fractional-threshold}, while 
Theorem \ref{thm:space-dominant} yields the exact
space-dominant branch.  Since the two limiting densities are unequal,
Theorem \ref{thm:exact-threshold} follows.

\subsection{A short history}

The exact codegree threshold was determined by R\"odl, Ruci\'nski and
Szemer\'edi~\cite{RRS09}, built on a few previous partial results.  Treglown and Zhao then settled all parameters
$d\ge k/2$~\cite{TZ12,TZ13} and formulated the general maximum
conjecture~\cite{TZ16}.  On the fractional side, Alon et
al.~\cite{AFHRRS} connected the problem to the Erd\H{o}s matching problem.
Subsequent progress of Han~\cite{Han16}, Frankl and
Kupavskii~\cite{FKup22}, and Lu and Yu~\cite{LuYu22} produced further exact ranges, including the fractional threshold for $d\ge 2k/5$ and the integer matching threshold for $d \ge 3k/8$; more recently, Frankl, Lu, Ma and Wu~\cite{FLMW26} settled the
additional cases $(k,d)=(5,1)$ and $(6,2)$.
The \textit{asymptotic} version of the minimum degree thresholds, that is, 
\[
 m_d(k,n)=
 \max\left\{
 1/2+o(1),
 \bigl(c^*_{k,d}+o(1)\bigr)\binom{n-d}{k-d}
 \right\},
\]
usually attributed to the H\`an--Person--Schacht conjecture and commonly regarded as the most important problem in the subject of hypergraph Dirac-type theory, was indeed implicitly resolved by the recent works of \cite{FHWYZZ26, NW26} and observation of Alon et
al.~\cite{AFHRRS}.
Please allow us omit the closely related (indeed, exciting) developments on partite versions, decision problems, random graph versions, counting versions, random sparsification versions, etc.


\subsection{Proof strategy}


Let us focus on the space-dominant case, as the other case follows from combining existing works.
The following definition is essential to us.
For $B\subseteq V(H)$, write
$e_H(B):=|E(H)\cap\binom Bk|$.

\begin{definition}\label{def:space-extremal}
Let $n\in k\mathbb N$.  
An $n$-vertex $k$-graph $H$ is
\emph{$\eps$-space-extremal} if some set $B\subseteq V(H)$ satisfies $|B|=\left(1-\frac1k\right)n$
and $e_H(B)\le\eps\binom nk$.
\end{definition}

For the space-dominant range, we split the proof into extremal and non-extremal cases. 
In the extremal case, we show that if $H$ is space extremal, a theorem of Lu, Yu and Yuan yields a perfect matching. 
In the non-extremal case, we first obtain an almost perfect matching via edge-disjoint fractional matchings, and a (standard) nibble-type result, and then absorb the uncovered vertices using the absorbing lemma of H\`an, Person and Schacht.

\section{Tools}\label{sec:prelim}

Let $H$ and $H'$ be $k$-graphs on the same vertex set $V$, where $|V|=n$.
In the stability arguments below, we shall compare a given hypergraph with an
appropriate extremal template. We therefore introduce both a global notion of
closeness and its corresponding local vertex version as follows.

\begin{definition}\label{def:close}
Let \(\varepsilon>0\). We say that \(H\) is
\emph{\(\varepsilon\)-contained in \(H'\)} if
$|E(H')\setminus E(H)|
\le \varepsilon n^k$.
\end{definition}

\begin{definition}\label{def:good-bad}
Let \(\alpha>0\). A vertex \(v\in V\) is said to be
\emph{\(\alpha\)-good with respect to \(H'\)} if
$
\bigl|
\{e\in E(H')\setminus E(H):v\in e\}
\bigr|
\le \alpha n^{k-1}$.
Otherwise, \(v\) is said to be
\emph{\(\alpha\)-bad with respect to \(H'\)}.
\end{definition}

We first introduce a key lemma established by Lu, Yu and Yuan~\cite{LYY21}.
If $U\dot\cup W$ is a partition of an $n$-vertex set and $1\le s\le k$, then let
\begin{equation}\label{eq:template}
  H_k^s(U,W):=
 \left\{e\in\binom{U\cup W}{k}:|e\cap W|\in[s]\right\}.
\end{equation}
We may omit the superscript $s$ if $s=k$.
Moreover, for simplicity, in some occasions we do not explicitly define the sets $U, W$ and write $H_k^s(|U|,|W|)$ instead.

The following result shows that any host $k$-graph $h$ close to $ H_k^s(U,W)$ with the sharp minimum $d$-degree condition must have a perfect matching.

\begin{lemma}[Lu--Yu--Yuan,~\cite{LYY21}]\label{good}
	Let $k,l, m,n$ be integers and $\alpha$ be a positive real, such that $k\ge 3$, $l\in [k-1]$,
         $\alpha<(8^{k-1}k^{5(k-1)}k!)^{-1}$,  $n\ge 8k^6$, and $n/(2k^5)\leq m \le  n/k$.
	Suppose that $H$ is a $k$-graph on $n$ vertices and $U,W$ is a partition of
        $V(H)$ with $|W|=m$ such that
	every vertex of $H$ is $\alpha$-good with respect to $H_k^{k-l}(U,W)$.
	Then $\nu (H)\ge m$.
\end{lemma}

\begin{lemma}[Lu--Yu--Yuan,~\cite{LYY21}]\label{lem:LYY-close}
Fix $k\ge3$ and $d\in[k-2]$.
There exist $\eps_0>0$ and $n_0$ such that the following holds for $n\ge n_0$.
Let $m=n/k-1$.
Suppose $H$ is an $n$-vertex $k$-graph and $U,W$ is a partition of $V(H)$ with $|W|=m$ such that $
 \delta_d(H)>
 \binom{n-d}{k-d}-\binom{n-d-m}{k-d}$
and $H$ is $\eps_0$-contained in $ H_k^{k-d}(U,W)$
then $H$ contains a perfect matching.
\end{lemma}

We now record the stability theorem that will provide the extremal structure
needed later.
Roughly speaking, it asserts that an almost extremal $k$-graph
with bounded matching number must be close to the corresponding space-barrier
construction.
We remark that this result built upon the results on Feige's conjecture \cite{FHWYZZ26, NW26}.

\begin{theorem}[Cao--Liu--Zhang,~\cite{CLZ26}]
\label{CLZ}
Fix $k\ge 2$, $0<a\le 1/(k+1)$, and $\eta>0$. There exist $\xi>0$ and $n_0$ such that the following holds whenever $n\ge n_0$ and $an\le m\le \frac{n}{k+1}$. If $\mathcal {F}\subseteq \binom{[n]}{k}$, $\nu(\mathcal {F})\le m$, and
\[
|\mathcal {F}|\ge {n\choose k}-{n-m\choose k}-\xi n^k,
\]
then $\mathcal{F}$ is $\eta$-contained in $H_1(n,m)$.
\end{theorem}

We next need the auxiliary tools for the passage from an almost perfect matching to a perfect matching. 
The first is an absorbing lemma of H\`an, Person and Schacht~\cite{HPS09}.

\begin{lemma}[Absorbing lemma, H\`an--Person--Schacht~\cite{HPS09}]\label{lem:absorber}
Given integers $d\in[k-1]$ and $\gamma\in (0,1)$, there exists $n_0$ such that the following holds for every $k$-graph $H$ on $n\ge n_0$ vertices.  
If $\delta_d(H)\ge\left(\frac12+2\gamma\right)\binom{n-d}{k-d}$,
then there exists a set $X\subseteq V(H)$ such that
$ |X|\le \gamma^kn$
and, for every $W\subseteq V(H)\setminus X$ satisfying $
 |W|\le \gamma^{2k}n$
and  $|W|\in k\mathbb N$,
there is a matching in $H$ covering exactly the vertices of $X\cup W$.
\end{lemma}

After reserving such an absorbing set, it remains to find a matching covering almost all vertices of the remaining hypergraph. 
For this purpose, we use the following standard result of Frankl and R\"odl~\cite{FR85}, which guarantees an almost perfect matching in an almost regular hypergraph with sufficiently small pair-degrees.

\begin{lemma}[Frankl--R\"odl,~\cite{FR85}]\label{lem:FR}
For every integer $k\ge2$ and every $\sigma>0$, there exist $\tau>0$ and $D_0$ such that the following holds.  Let $J$ be an $n$-vertex $k$-graph and let $D$ satisfy $n\ge D\ge D_0$.  
If for each vertex $v\in V(J)$, we have $d(v)=(1\pm\tau)D$
and for every pair $\{u,v\}$ of vertices of $J$, we have
$d(u,v)<\tau D$,
then $J$ contains a matching covering all but at most $\sigma n$ vertices.
\end{lemma}

We shall use the following rainbow fractional matching theorem of
Aharoni, Holzman, and Jiang~\cite{AHJ}.
Roughly speaking, it allows us to select one edge from each member of a family of hypergraphs while preserving a prescribed fractional matching size.

\begin{theorem}[Aharoni--Holzman--Jiang,~\cite{AHJ}]\label{rainfracmat}
    Let $r\ge 2$ be an integer, and let $n$ be a positive rational number. Let
    $H_1,\dots,H_{\lceil rn\rceil}$ be $r$-graphs such that $\nu^*(H_i)\ge n$ for
    $i\in[\lceil rn\rceil]$. Then there exist $e_1\in H_1,\dots,e_{\lceil rn\rceil}\in
    H_{\lceil rn\rceil}$ such that $\left\{e_1,\dots,e_{\lceil rn\rceil}\right\}$ has a fractional
    matching of size $n$.
\end{theorem}

We also need the following standard concentration estimates.

\begin{lemma}[Chernoff]\label{chernoff}
Let $X$ be a sum of independent Bernoulli random variables with
$\mu=\mathbb E X$. Then, for every $0<\delta\le1$,
\[
\Pr(|X-\mu|\ge \delta\mu)
\le 2\exp\left(-\frac{\delta^2\mu}{3}\right).
\]
Moreover, for every $t>\mu$,
\[
\Pr(X\ge t)
\le
\left(\frac{e\mu}{t}\right)^t.
\]
\end{lemma}




\section{The extremal case}\label{sec:extremal}

Now we are ready to state and prove our extremal case theorem.

\begin{theorem}[Extremal case]\label{thm:extremal}
Fix $k\ge3$ and $1\le d\le k-2$.  There exist $\eta>0$ and $n_0$ such that the following holds whenever $n\ge n_0$ and $k\mid n$.  If an $n$-vertex $k$-graph $H$ satisfies
\[
 \delta_d(H)\ge s_d(k,n)+1
\]
and is $\eta$-space-extremal, then $H$ has a perfect matching.
\end{theorem}

\begin{proof}
Put $r=k-d$ and $q=n/k$.  Let $\eps_0$ and $N_0$ be supplied by Lemma \ref{lem:LYY-close}.  Choose $\eta \ll \eps_0$.

Since $H$ is $\eta$-space-extremal, there exists $B_0\subseteq V(H)$ such that
$|B_0|=n-q$ and $e_H(B_0)\le\eta\binom nk$.
Let $A_0:=V(H)\setminus B_0$, choose any $x\in A_0$, and define
$W:=A_0\setminus\{x\}$ and $U:=B_0\cup\{x\}$.
Then $|W|=q-1$ and $|U|=n-q+1$.
By Lemma~\ref{lem:LYY-close}, it suffices to show that $H$ is $\eps_0$-contained in $\mathcal T:= H_k^r(U,W)$, that is, $|\mathcal T\setminus H|\le\eps_0 n^k$.



Fix $D\in\binom Ud$.
We shall consider $b_D:=\deg_{\mathcal T\setminus H}(D)$,
$c_D:=\deg_{H[U]}(D)$, and $\deg_{\mathcal T}(D)$.
The number of $\mathcal T$-edges containing $D$ is
\[
 \deg_{\mathcal T}(D)
 =\binom{n-d}{r}-\binom{|U|-d}{r}
 =\binom{n-d}{r}-\binom{n-d-q+1}{r}
 =s_d(k,n).
\]
Every $H$-edge containing $D$ is either contained in $U$, or meets $W$ in between one and $r$ vertices and hence belongs to $\mathcal T$. So we have the identity
 $\deg_H(D)=s_d(k,n)-b_D+c_D$.
The minimum-degree assumption gives
$ c_D = \deg_H(D)-s_d(k,n)+b_D \ge b_D$.

As $U=B_0\cup\{x\}$ and $n$ is large enough, we have
\[
 e_H(U)
 \le e_H(B_0)+\binom{|B_0|}{k-1}
 \le\eta\binom nk+\binom n{k-1} \le 2\eta\binom nk.
\]
Now sum over all $D\in\binom Ud$.  
Since $|e\cap W|\le r=k-d$, we have $|e\cap U|\ge d$, so every missing $\mathcal T$-edge is counted at least once.  Hence, as $\eta \ll \eps_0$, we get
\[
 |\mathcal T\setminus H|
 \le\sum_{D\in\binom Ud}b_D
 \le\sum_{D\in\binom Ud}c_D
 =\binom kd e_H(U) \le \eps_0 n^k.
\]
Thus $H$ is $\eps_0$-contained in $ H_k^r(U,W)$ and thus \cref{lem:LYY-close} yields a perfect matching in $H$.
%
\end{proof}

\section{Almost Perfect Matching}

We need the following elementary result on the density of sets near extremal size.
\begin{lemma}\label{lem:robust-density}
Let $N\in k\mathbb N$, set $\beta:=1-1/k$, and let $H$ be an $N$-vertex $k$-graph
which is not $\eps$-space-extremal.  If $0<\zeta\le\eps/(2k)$, then every
set $A\subseteq V(H)$ satisfying $|A|\ge(\beta-\zeta)N$
spans at least $e_H(A)\ge(\eps-k\zeta)\binom Nk
 \ge\frac\eps2\binom Nk$
edges.
\end{lemma}

\begin{proof}
By adding up to $\zeta N$ vertices if necessary, let $B\supseteq A$ be a set of size at least $\beta N$.
As $H$ is not $\eps$-space-extremal, we obtain that 
\[
\eps \binom Nk\le  e_H(B)\le e_H(A)+\zeta N\binom{N-1}{k-1}
 =e_H(A)+k\zeta\binom Nk,
\]
which gives the result.
\end{proof}


We remark that the combination of the following two results is the main ingredient for determining the sharp minimum degree thresholds for perfect matchings which has been missing for the past nearly two decades, and is possible now using Theorem~\ref{CLZ}.
The following result says that a non-space-extremal $k$-graph with slightly weakened minimum $d$-degree conditions contains a perfect fractional matching.
{For perfect fractional matchings, it is convenient to work with the following notion of monotone hypergraphs.

{Given a $k$-graph $H$ defined on $[n]$, for every two $k$-sets $S=\{s_1,\ldots,s_k\}$ and $S'=\{s_1',\ldots,s_k'\}$ where $s_1\le\cdots\le s_k$ and $s_1'\le\cdots\le s_k'$, write $S\preceq S'$ if
$s_i\le s_i'$ for every $i\in[k]$.
A $k$-graph $H$ on $[n]$ is called \emph{monotone} 
if $S'\in E(H)$ and $S\preceq S'$ imply $S\in E(H)$.
Equivalently, whenever $e\in E(H)$, $j\in e$, $i\notin e$,
and $i<j$, we have
$(e\setminus\{j\})\cup\{i\}\in E(H)$.}

\begin{lemma}\label{main-frac}
Fix integers \(k\ge3\) and \(d\in[k-2]\) with $ (1-1/k)^{k-d} < 1/2$, and let
$0<\beta\ll\alpha\ll\varepsilon,\gamma,1/k$.
Then, for every sufficiently large \(n\in k\mathbb N\), the following
holds.
Let \(H\) be a {monotone} \(k\)-graph on vertex set \([n]\) which is not
\(\varepsilon\)-space-extremal. Suppose that
$\delta_1(H)>\gamma n^{k-1}$,
and that there exists a set \(S\subseteq V(H)\) with
$|S|\le\beta n$
such that
\[
\delta_d(H-S)
\ge
\binom{n-d}{k-d}
-
\binom{n-n/k}{k-d}
-
\alpha n^{k-d}.
\]
Then \(H\) admits a perfect matching.
\end{lemma}

\begin{proof}
Choose \(0<\eta\ll\varepsilon,\gamma,1/k\).
Applying Theorem~\ref{CLZ} with  \(r=k-d\),
\(a=1/(2k)\), and parameter \(\eta\), let
\(\xi>0\) denote the corresponding constant
 supplied by Theorem~\ref{CLZ}.
Choose
\[
1/n\ll\beta\ll\alpha\ll\xi\ll\eta
\ll\varepsilon,\gamma,1/k.
\]
Let $H$ be a monotone $k$-graph on $[n]$ with properties stated in the lemma.

 \medskip
\begin{claim}  
\(H\) admits a matching \(M_0\) that covers the set \(S\).
\end{claim}
\begin{proof}
We greedily construct such a matching in $H$.
Suppose that $M$ is the matching constructed so far,
with every edge of $M$ intersecting $S$, and that $M$ does not yet cover all
vertices of $S$. 
Choose a vertex $x\in S\setminus V(M)$.
Since \(|M|\le |S|\le \beta n\), we have $|V(M)|\le k\beta n$.
The number of edges containing \(x\) and intersecting \(V(M)\) is at most
$|V(M)|\binom{n-2}{k-2}
\le k\beta n^{k-1}
<\gamma n^{k-1}
<d_H(x)$ since $\beta\ll\gamma$.
Hence there is an edge \(e\in E(H)\) containing \(x\) and disjoint from
\(V(M)\). Add \(e\) to \(M\) and continue.
At each step at least one new vertex of \(S\) is covered, so the procedure
terminates after at most \(|S|\) steps and yields a matching covering \(S\).
\end{proof}

Let $l:=|M_0|$. 
Let $H_1:=H-V(M_0)$ and let $D\subseteq V(H_1)$ consist of the $d$ vertices of $H_1$ with largest labels. Let $G:=N_{H_1}(D)$, which is a $(k-d)$-graph on a set of $n-kl-d$ vertices.

\begin{claim}
\label{clm:clpm}
If $\nu(G)\geq n/k-l$, then $H$ contains a perfect matching.
\end{claim}
\begin{proof}
Let $M:=\{e_1'\ldots,e_{n/k-l}'\}$ be a matching of size $n/k-l$ of $G$.  
We partition $[n]\setminus V(M_0\cup M)$ into $n/k-l$ vertex-disjoint $d$-subsets, $e_1,\ldots,e_{n/k-l}$. 
Because vertices in $D$ have the smallest weights, by the monotonicity of $H$, we have $e_i\cup e_i'\in E(H)$ for all $i\in [n/k-l]$.
Then $M_0\cup \{e_i\cup e_i'\ :\ i\in [n/k-l]\}$ is a perfect matching of $H$.  
\end{proof}

By Claim \ref{clm:clpm}, we may assume that $\nu(G)\leq n/k-l-1$ and shall apply Theorem~\ref{CLZ} to $G$. 
Fix $m:=n/k-l-1$, $r:=k-d$ and $N:=n-kl-d$. 
Since \(l\le |S|\le\beta n\), for sufficiently large \(n\),
$N/2k\le m\le N/k\le N/(r+1)$.
Moreover, since \(S\subseteq V(M_0)\) and \(1/n\ll\beta\ll\alpha\ll\xi\), we have
\begin{align*}
e(G)=\deg_{H_1}(D)
\ge \delta_d(H-S)
   -|V(M_0)\setminus S|\binom{n-d-1}{r-1}
\ge
\binom Nr-\binom{N-m}{r}-\xi N^r.
\end{align*}
Therefore, Theorem~\ref{CLZ} implies that
\(G\) is \(\eta\)-contained in \(H_r(N,m)\) with an underlying vertex partition $U\dot\cup W$ with $|U|=m$. 
By averaging, $G$ contains at most $r\eta^{1/2}N$ $\eta^{1/2}$-bad  vertices with respect to $H_{r}(N,m)$,  denoted by $B$. 

Choose
$U'\subseteq U\setminus B$ with $
|U'|=n/k-r\eta^{1/2}n-l$. Then we have $|V(H_1)\setminus U'|\ge \left(1-1/k\right)n$.
By Lemma \ref{lem:robust-density}, we have $e(H_1-U')\ge \frac{\eps}{2}\binom{n-kl}{k}$, which yields a matching $M_1$ of size $r\eta^{1/2}n$ in $H_1-U'$ via greedy choices. 
Now every vertex of $G-V(M_1)-B$ is $2\eta^{1/2}$-good with respect to $H_{r}(N-|B\cup V(M_1)|,|U'|)$. 
So we can apply Lemma \ref{good} to $G-V(M_1)-B$ and get a matching $M_2$ of size $t:=n/k-r\eta^{1/2}n-l$. 
Let $M_2:=\{h_1,\ldots,h_t\}$ and partition  $[n]-V(M_0\cup M_1\cup M_2)$ arbitrarily into $d$-sets $h_1',\ldots,h_t'$. 
Let $M_2':=\{h_i\cup h_i': i\in [t]\}$, which is a matching in $H$ by definition.
 It follows that $M_0\cup M_1\cup M_2'$ is a perfect matching of $H$. This completes the proof. \end{proof}

The following lemma helps us drop the monotonicity assumption at the price of getting an almost perfect matching, whose idea originated from Han--Kohayakawa--Person~\cite{HKP21}.
\begin{lemma}\label{lem:almost-perfect}
Fix integers \(k\ge3\) and \(d\in[k-2]\) with with $ (1-1/k)^{k-d} < 1/2$, and let
$0<\rho_0\ll\rho_1\ll\rho_2\ll 1/k$.
Then, for every sufficiently large \(n\in k\mathbb N\), the following
holds. Let \(H\) be an \(n\)-vertex \(k\)-graph such that \[\delta_d(H)>\binom{n-d}{k-d}
-\binom{n-n/k+1-d}{k-d}-\rho_1 n^{k-d}.\]
If \(H\) is not \(\rho_2\)-space-extremal, then $H$ admits a matching covering all but at most
$\rho_0 n$ vertices.
\end{lemma}

\begin{proof}  We start with the following claim.
\begin{claim}
\(H\) contains \(\lfloor n/\ln n\rfloor\) edge-disjoint perfect fractional matchings \(f_1,\dots,f_{\lfloor n/\ln n\rfloor}\) satisfying
$
\sum_{i=1}^{[\lfloor n/\ln n\rfloor]}\sum_{\{x,y\}\subseteq e} f_i(e) \le 3$
for every pair \(\{x,y\}\in \binom{V(H)}{2}\).
\end{claim}
\begin{proof}

Let $q=\lfloor n/\ln n\rfloor$.
We construct the perfect fractional matchings $f_1,\ldots,f_q$ inductively.
Set $E_0=T_0=\emptyset$.
For $s\ge0$, suppose that $f_1,\ldots,f_s$ have already been constructed, where for $s=0$ this denotes the empty family, such that
$|E_i|\le n$ for every $i\in[s]$ and
$L_s(A):=\sum_{i=1}^s\sum_{A\subseteq e}f_i(e)\le3$
for every $A\in\binom{V(H)}2$, where
$E_i:=\{e\in E(H):f_i(e)>0\}$.
Set $T_s:=\{A\in\binom{V(H)}2:L_s(A)>2\}$ and
$E'_s:=\{e\in E(H):A\subseteq e\text{ for some }A\in T_s\}$.
Define $H_s$ by
$E(H_s):=E(H)\setminus\bigl((\bigcup_{i=1}^sE_i)\cup E'_s\bigr)$.
In particular, $H_0=H$.
We shall show that, for every $0\le s<q$, $H_s$
admits a perfect fractional matching $f_{s+1}$.

{
Let $\omega:V(H_s)\rightarrow [0,1]$ be a minimum fractional vertex cover of $H_s$ with $\omega(1)\ge\omega(2)\ge\cdots\ge\omega(n)$ (we relabel $V(H)$ if needed). 
Define the monotone \(k\)-graph \(cl(H_s)\) on vertex set \(V(H)\) by
\[
E(cl(H_s))
:=
\left\{
e\in\binom{V(H)}k:\omega(e)\ge1
\right\}.
\]
Note that 
i) $\omega$ is also a fractional vertex cover of $cl(H_s)$, that is, $\nu^*(cl(H_s))=\tau^*(cl(H_s))\le \sum_{x\in [n]}\omega(x)=\nu^*(H_s)$, where we use the duality theorem, and 
ii) $H_s$ is a subgraph of $cl(H_s)$, giving $\nu^*(H_s)\le \nu^*(cl(H_s))$.
Putting together, we obtain $\nu^*(H_s)=\nu^*(cl(H_s))$.
}

We show that \(cl(H_{s})\) satisfies the other assumptions of Lemma \ref{main-frac} and thus has a perfect matching.
First, we have $\sum_{i\in [s]}|E_i|\le sn\le n^2/\ln n$.
Viewing $T_s$ as a graph, its maximum degree is at most $(k-1)s/2\le ks$ -- for $v\in V(H)$, the sum of weights of pairs containing $v$ is upper bounded by $(k-1)s$, because each edge weight is counted $k-1$ times towards the sum; thus the number of pairs with weight larger than 2 is at most $s(k-1)/2$. 
Therefore, every vertex is incident to at most $s(k-1)n\binom{n-3}{k-3}/2+s(k-1)\binom{n-2}{k-2}/2\le kn^{k-1}/\ln n$ edges in $E_s'$, and thus $|E_s'|\le n^{k}/\ln n$.
So we obtain $|\bigcup_{i=1}^s E_i\cup E_s'|\le n^2/\ln n+n^{k}/\ln n$ and $\delta_1(cl(H_s))\ge \delta_1(H_s)\ge \delta_1(H) - n^2/\ln n - n^{k-1}/\ln n \ge \binom{n-1}{k-1}/k$.
Note that
$|E(H)\setminus E(cl(H_s))|\le |\bigcup_{i=1}^s E_i\cup E_s'| \le n^2/\ln n+ n^k/\ln n$.
As $H$ is not $\rho_2$-space-extremal, \(cl(H_{s})\) is not $\rho_2/2$-space-extremal.

It remains to verify the $d$-degree assumption.
Let $Q:=\{n-\rho_0n+1,\ldots,n\}$. 
Note that $|T_s|\leq ksn\le kn^2/\ln n$. 
For $d\geq 2$, the number of $d$-subsets $R$ with ${R\choose 2}\cap T_s\neq \emptyset$ is at most ${n-2\choose d-2}k^2n^2/\ln n<{|Q|\choose d}$. 
So there exists $A\in {Q\choose d}$ such that ${R\choose 2}\cap T_s=\emptyset$. 
Now we claim
\[
\delta_d\bigl(cl(H_s)-(Q\setminus A)\bigr)
>
\binom{n-d}{k-d}
-\binom{n-n/k+1-d}{k-d}
-2\rho_1n^{k-d}.
\]
%
Indeed, note that \(Q\) consists of the vertices of smallest weights
and \(A\in\binom Qd\).
By the definition of \(cl(H_s)\), this implies
$\delta_d\bigl(cl(H_s)-(Q\setminus A)\bigr)
=
d_{cl(H_s)-(Q\setminus A)}(A) \ge d_{H_s-(Q\setminus A)}(A)$.
Since \(\binom A2\cap T_s=\emptyset\), every pair in \(T_s\)
contained in an edge containing \(A\) either meets \(A\) in exactly one
vertex or is disjoint from \(A\). 
Hence
\begin{align*}
d_H(A)-d_{H_s-(Q\setminus A)}(A)
\le
|Q|\binom{n-d-1}{k-d-1}
+\sum_{i=1}^s|E_i|
+
ksd\binom{n-d-1}{k-d-1}
+
|T_s|\binom{n-d-2}{k-d-2},
\end{align*}
where the four terms correspond respectively to edges lost by
deleting \(Q\setminus A\), edges belonging to some \(E_i\), edges
containing a pair of \(T_s\) meeting \(A\) in one vertex, and edges
containing a pair of \(T_s\) disjoint from \(A\).
Since
$|Q|=\rho_0n$,
$s\le n/\ln n$,
$\sum_{i=1}^s|E_i|\le sn$,
$|T_s|\le ksn$,
we obtain
$
d_H(A)-d_{H_s-(Q\setminus A)}(A)
<
\rho_1n^{k-d}$.
 The claim follows from the lower bound
on \(\delta_d(H)\).

Applying Lemma \ref{main-frac}, \(cl(H_{s})\) has a  perfect matching, which yields that $H_s$ has a perfect fractional matching \(f_{s+1}\). By Theorem~\ref{rainfracmat}, we may choose $f_{s+1}$ such that 
\[\left|\{e\in E(H_s):f_{s+1}(e)>0\}\right|<n.\]
Repeating this iterative procedure, we obtain \(n/\ln n\) edge-disjoint perfect fractional matchings with the pair condition preserved.
\end{proof}

Returning to the proof of the lemma, let \(h:=\sum_{i=1}^{n/\ln n} f_i\). 
We construct a random subgraph \(R\subseteq H\) by including each edge \(e\in E(H)\) with probability \(h(e)\). 
The disjointness guarantees that $h(e)\in [0,1]$ for all $e$.
For any vertex \(x\in V(H)\), we have \(\mathbb{E}[\deg_R(x)] = n/\ln n\); for any distinct pair \(x,y\in V(H_1)\), \(\mathbb{E}[\deg_R(\{x,y\})] \le 3\). 
By Lemma \ref{chernoff}, with probability \(1-o(1)\) there exists a spanning subgraph \(F\subseteq H\) satisfying:
\begin{itemize}
    \item For every vertex \(x\in V(H_1)\), \(\deg_F(x) \sim n/\ln n\) and \(\mathbb{E}[\deg_F(x)] = n/\ln n\);
    \item For every pair of distinct vertices \(x,y\in V(H_1)\), \(\deg_F(\{x,y\}) \le \ln n\).
\end{itemize}
By Lemma \ref{lem:FR}, \(H\) contains a matching \(M_1\) such that \(|V(H)\setminus V(M_1)|\le \rho_0 n\). 
\end{proof}

\section{The non-extremal case}\label{sec:nonextremal}

The main purpose of this section is to establish the following result.

\begin{theorem}[Non-extremal case]\label{thm:nonextremal}
Fix integers $k\ge3$ and $d\in[k-1]$, and assume that $c=1-\left(1-\frac1k\right)^{k-d}>\frac12$.
For every $\eps>0$, there exist $\gamma_{\mathrm{ne}}>0$ and $n_0$ such that the following holds for $n\ge n_0$ and $n\in k\mathbb N$.  
If $H$ is an $n$-vertex $k$-graph $H$ such that $\delta_d(H)\ge(c-\gamma_{\mathrm{ne}})
 \binom{n-d}{k-d}$
and is not $\eps$-space-extremal, then $H$ contains a perfect matching.
\end{theorem}

\begin{proof}
Set $r:=k-d$, $\beta:=1-1/k$, and $\Delta:=c-\frac12>0$. 
Fix a constant $C_{k,d}$ such that, whenever $0\le q\le n/2$, $N=n-q$, and $n$ is sufficiently large, the following holds.
\begin{equation}\label{eq:universal-deletion-constant}
 \binom{n-d}{r-1}
 \le \frac{C_{k,d}}n\binom{N-d}{r}.
\end{equation}
Such a constant exists because $r$ is fixed and $N\ge n/2$.


Let $1/n\ll\gamma_{\mathrm{ne}}\ll\theta\ll \delta\ll\Delta,\eps,1/k,1/d\ll1$.
We have $\delta_d(H)\ge\left(\frac12+\theta\right)
 \binom{n-d}{k-d}$ since $\gamma_{\mathrm{ne}}\ll\theta\ll\Delta$.
Applying Lemma \ref{lem:absorber} with $\theta/2$ in place of $\gamma$, we obtain a set $X\subseteq V(H)$ with $|X|\le(\theta/2)^kn$ 
and, for every $W\subseteq V(H)\setminus X$ satisfying $
 |W|\le (\theta/2)^{2k}n$
and  $|W|\in k\mathbb N$,
there is a matching in $H$ covering exactly the vertices of $X\cup W$.  

Let $N:=n-|X|$ and $H_0=H-X$, it is easy to see that $|X|\in k\mathbb N$ and thus $N\in k\mathbb N$.
As $|X|\le (\theta/2)^k n$, we have $|X|\binom{n-d-1}{k-d-1}\le \theta^k n^{k-d}$.
By $\gamma_{\mathrm{ne}}\ll\theta\ll\delta$, we obtain
 \[
 \delta_d(H_0)
 \ge(c-\gamma_{\mathrm{ne}})
 \binom{N-d}{r} - |X|\binom{n-d-1}{k-d-1}
 \ge(c-\delta)
 \binom{N-d}{r}.
 \]
Next we show that $H_0$ is not $(\eps/2)$-space-extremal.  Otherwise there would exist $B_0\subseteq V(H_0)$ such that
 $|B_0|=\beta N$, and $e_{H_0}(B_0)\le\frac\eps2\binom Nk$.
 This contradicts Lemma \ref{lem:robust-density}.


Apply Lemma \ref{lem:almost-perfect} to $H_0$ with $ (\theta/2)^{2k},\delta,\eps/2$ in place of $\rho_0,\rho_1, \rho_2$. 
We obtain a matching $M$ in $H_0$ whose uncovered set
$W:=V(H_0)\setminus V(M)$
satisfies $|W|\le( \theta/2)^{2k}N\le  (\theta/2)^{2k}n$.  
Since $N\in k\mathbb N$, we also have $|W|\in k\mathbb N$.  By Lemma \ref{lem:absorber}, there is a matching covering exactly $X\cup W$.  Together with $M$, this forms a perfect matching of $H$.
\end{proof}

\section{Proofs of the main theorems}\label{sec:main-proof}

\begin{proof}[Proof of Theorem \ref{thm:space-dominant}]
Let $r=k-d$ and $c=1-(1-1/k)^r>1/2$.  First observe that $d\le k-2$: if $d=k-1$, then $r=1$ and $c=1/k\le1/3$.

\smallskip\noindent\emph{Lower bound.}
The space construction described in the introduction gives the lower bound
$m_d(k,n)\ge s_d(k,n)+1$.
Indeed, take $A\subseteq[n]$ with $|A|=n/k-1$ and let $H_{\mathrm{sp}}$ consist of all $k$-sets meeting $A$.  Every matching has size at most $|A|<n/k$, while
$\delta_d(H_{\mathrm{sp}})=s_d(k,n)$.

\smallskip\noindent\emph{Upper bound.}
Let $\eta>0$ be supplied by Theorem \ref{thm:extremal}.  Apply Theorem \ref{thm:nonextremal} with $\eps=\eta$, obtaining a constant $\gamma_{\mathrm{ne}}>0$.  Since
\[
 \frac{s_d(k,n)}{\binom{n-d}{r}}
 =1-
 \frac{\binom{\beta n+1-d}{r}}{\binom{n-d}{r}}
 =c+O_{k,d}(n^{-1}),
\]
we have, for all sufficiently large $n$,
\[
 \delta_d(H)\ge s_d(k,n)+1
 \ge(c-\gamma_{\mathrm{ne}})
 \binom{n-d}{r}.
\]

Let $H$ be an $n$-vertex $k$-graph with $\delta_d(H)\ge s_d(k,n)+1$.
If $H$ is $\eta$-space-extremal, then Theorem \ref{thm:extremal} gives a perfect matching; if $H$ is not $\eta$-space-extremal, then Theorem \ref{thm:nonextremal} gives a perfect matching.
Therefore we obtain
 $m_d(k,n)\le s_d(k,n)+1$.
Together with the lower bound, this proves the theorem.
\end{proof}

\begin{proof}[Proof of Theorem \ref{thm:exact-threshold}]
Put $r=k-d$.  As observed in the introduction, the equality
$c_{k,d}^{\mathrm{sp}}=1/2$ is impossible for $k\ge3$.  If
$c_{k,d}^{\mathrm{sp}}>1/2$, then Theorem \ref{thm:space-dominant} and
\eqref{eq:barrier-asymptotics} give
\[
 m_d(k,n)=s_d(k,n)+1
 =\max\{s_d(k,n),\ddiv(n,k,d)\}+1
\]
for all sufficiently large $n\in k\mathbb N$.  If
$c_{k,d}^{\mathrm{sp}}<1/2$, then
Theorem \ref{cor:fractional-threshold} and the result of Treglown and
Zhao~\cite[Theorem 2]{TZ16} give
$m_d(k,n)=\ddiv(n,k,d)+1$.
In this case \eqref{eq:barrier-asymptotics} gives
$\ddiv(n,k,d)>s_d(k,n)$ for all sufficiently large $n$.  The two cases prove
the stated maximum formula.
\end{proof}

\section{Closing remarks}

The hypergraph perfect matching conjectures, both asymptotic and exact ones, received a lot of attention in the past two decades, and indeed inspired not only a long line of research on hypergraph Dirac-type problems, but also a (small) group of mathematicians working extensively on this thread of problems.
Proudly being members of them, the authors enjoyed and appreciated the time spent on the matching and other related problems.




\begingroup
\raggedright

\endgroup

\end{document}